\documentclass[a4paper,11pt]{article}
\usepackage{amssymb}
\usepackage{amsthm}
\usepackage{tikz}
\usepackage{tikz,pgfplots}
\usetikzlibrary{decorations.markings}
\usepackage{amsmath,amssymb}
\usepackage{todonotes}

\usepackage{color}
\usepackage{graphicx}
\usepackage{placeins,caption} 

\DeclareMathOperator{\mob}{Mob_{gp}}
\DeclareMathOperator{\gp}{gp}
\usepackage[top=3cm,bottom=3cm,left=2.8cm,right=2.8cm]{geometry}
\newcommand{\address}[1]{#1}
\newtheorem{theorem}{Theorem}[section]
\newtheorem{lemma}[theorem]{Lemma}

\theoremstyle{definition}

\newcommand{\move}{\rightsquigarrow}

\tikzset{middlearrow/.style={
		decoration={markings,
			mark= at position 0.75 with {\arrow[scale=2]{#1}} ,
		},
		postaction={decorate}
	}
}

\tikzset{midarrow/.style={
		decoration={markings,
			mark= at position 0.75 with {\arrow[scale=2]{#1}} ,
		},
		postaction={decorate}
	}
}

\begin{document}
	
	\title{Traversing a graph in general position}
	\author{Sandi Klav\v{z}ar $^{a,b,c}$ \\ \texttt{\footnotesize sandi.klavzar@fmf.uni-lj.si}  \\
		{\footnotesize ORCID: 0000-0002-1556-4744}
		\and
		Aditi Krishnakumar $^{d}$ \\ \texttt{\footnotesize aditikrishnakumar@gmail.com} \\
		{\footnotesize ORCID: 0000-0002-0954-7542}
		\and 
		James Tuite $^{d}$ \\ \texttt{\footnotesize james.t.tuite@open.ac.uk} \\ 
		{\footnotesize ORCID: 
			0000-0003-2604-7491}
		\and 
		Ismael G. Yero $^{e}$ \\ \texttt{\footnotesize ismael.gonzalez@uca.es} \\
		{\footnotesize ORCID: 0000-0002-1619-1572}
	}
	
	\maketitle
	
	\address{
		$^a$ Faculty of Mathematics and Physics, University of Ljubljana, Slovenia
		
		$^b$ Institute of Mathematics, Physics and Mechanics, Ljubljana, Slovenia
		
		$^c$ Faculty of Natural Sciences and Mathematics, University of Maribor, Slovenia
		
		$^d$ Department of Mathematics and Statistics, Open University, Milton Keynes, UK
		
		$^e$ Departamento de Matem\'aticas, Universidad de C\'adiz, Algeciras, Spain 
	}
	
	\begin{abstract}
		Let $G$ be a graph. Assume that to each vertex of a set of vertices $S\subseteq V(G)$ a robot is assigned. At each stage one robot can move to a neighbouring vertex. Then $S$ is a mobile general position set of $G$ if there exists a sequence of moves of the robots such that all the vertices of $G$ are visited whilst maintaining the general position property at all times. The mobile general position number of $G$ is the cardinality of a largest mobile general position set of $G$. In this paper, bounds on the mobile general position number are given and exact values determined for certain common classes of graphs including block graphs, rooted products, unicyclic graphs, Kneser graphs $K(n,2)$, and line graphs of complete graphs. 
	\end{abstract}
	
	
	\noindent
	AMS Subj.\ Class.\ (2020): 05C12, 05C76
	
	\section{Introduction}
	
	In this article $G=(V(G),E(G))$ will represent a connected simple graph whose {\em order} is $n(G)=|V(G)|$. We will indicate that vertices $u$ and $v$ are adjacent by writing $u \sim v$. A \emph{$u,v$-path} of length $\ell $ is a sequence $u=u_0,u_1,\dots ,u_{\ell -1},u_{\ell }=v$ of distinct vertices of $G$ such that $u_i \sim u_{i+1}$ for $0 \leq i < \ell$.  The {\em distance} $d_G(u,v)$ between two vertices $u,v\in V(G)$ is the length of a shortest $u,v$-path. A {\em clique} of $G$ is a set $S\subseteq V(G)$ of mutually adjacent vertices, i.e. $S$ induces a complete graph. The {\em clique number}, denoted $\omega(G)$, is the cardinality of a largest clique in $G$. For a given set $S\subset V(G)$, the subgraph induced by $S$ will be written $G[S]$. For a positive integer $k$ we will use the notation $[k]=\{1,\dots,k\}$.
	
	General position sets in graphs have been widely studied in recent years, for example in~\cite{AnaChaChaKlaTho, ghorbani-2021, klavzar-2021, patkos-2020, tian-2021a, tian-2021b, yao-2022}. The concept was independently introduced in~\cite{manuel-2018} and~\cite{ullas-2016}; the terminology and set-up of the present work follow the former paper. However, the general position sets of hypercubes and integer lattices were investigated earlier in a different context in~\cite{korner-1995, palvolgyi-2014}, respectively. 
	
	One of the original motivations of the general position problem in~\cite{manuel-2018} was to place a set of robots in a graph such that any pair of robots situated at vertices $u,v$ can exchange signals by any shortest $u,v$-path without being obstructed by another robot. This static concept can be transformed into a dynamic one which is more closely related to practical problems in robotic navigation and transport. In fact this research was inspired by the delivery robots belonging to Starship Technologies\textregistered\ \cite{starship} that deliver groceries to the inhabitants of cities including Milton Keynes, home of the Open University. For some related studies on robot mobility in computer science see~\cite{aljohani-2018a,aljohani-2018b, diluna-2017}. In this paper we introduce a variant that describes the largest number of robots that can travel through a network such that each vertex of the network can be visited by a robot, whilst at every stage any pair of robots can see each other through any shortest path between their positions without being obstructed by another robot. We now describe this concept in greater detail.
	
	A set of vertices $S$ of a graph $G$ forms a {\em general position set} if no three distinct vertices from $S$ lie on a common shortest path. The {\em general position number} $\gp (G)$ of $G$ is the cardinality of a largest general position set. Assume that to each vertex of a general position set $S\subseteq V(G)$ one robot is assigned. The robots can move through the graph one at a time. We say that a move of a robot is {\em legal} if the robot moves to an adjacent unoccupied vertex such that the new set of occupied vertices is also in general position. If there exists a sequence of legal moves such that every vertex of $G$ can be visited by at least one robot, then we say that $S$ is a {\em mobile general position set}. The \emph{mobile general position number} $\mob(G)$ of $G$ is the cardinality of a largest mobile general position set of $G$. Such a set will be briefly called a {\em mobile gp-set} of $G$. We will also abbreviate the term mobile general position set to {\em mobile set} and mobile general position number to {\em mobile number}. A move of a robot from vertex $u$ to a neighbor $v$ will be denoted by $u \move v$.
	
	The paper is organised as follows. In the rest of the introduction we give some preliminary results and examples. In Section~\ref{sec:cut} we consider graphs with cut vertices, in particular block graphs, rooted products, and unicyclic graphs. In Section~\ref{sec:kneser} the mobile number is determined for Kneser graphs $K(n,2)$ and line graphs of complete graphs. We conclude the paper with some open problems. 
	
	\subsection{Preliminaries}
	\label{sec:prelim}
	
	We begin with the following bounds. 
	
	\begin{lemma}\label{lem:trivial bounds}
		If $G$ is a graph with $n(G)\ge 2$, then $2\le \mob(G) \leq \gp(G)$. Moreover, for any $2 \leq a \leq b$ there exists a graph with $\mob (G) = a$ and $\gp (G) = b$.
	\end{lemma}
	
	\begin{proof}
		As the set of occupied vertices at any stage must be a general position set of $G$, we have $\mob(G) \leq \gp(G)$. Any set of two vertices is in general position. Let $S = \{ u,v\} \subset V(G)$ be any set of two vertices of $G$; then for any vertex $w$ of $G$ the robot closest to $w$ (say $d_G(u,w) \leq d_G(v,w)$) can move along a shortest path to $w$ without crossing the robot at vertex $v$, so that $\mob(G) \geq 2$.
		
		To prove the second assertion, we claim that if $r_1 \geq \cdots \geq r_t$, where $t \geq 2$ and $r_1 \geq 2$, then 
		\begin{equation}
			\label{eq:complete-multi}
			\mob(K_{r_1,\dots ,r_t}) = \max \{2,t-1\}\,.
		\end{equation}
		Let $S$ be the set of occupied vertices in the initial configuration of $\mob(K_{r_1,\dots ,r_t}) $ robots. It is trivial that two robots can visit every vertex and remain in general position, so assume that $|S| \geq 3$. Assume that $S$ contains at least $2$ vertices from the same partite set $W$. Hence, there can be no robots in the other partite sets and, furthermore, if there are at least $3$ robots, then no robot can move from $W$ to a different partite set. Therefore, if $\mob(K_{r_1,\dots ,r_t}) \geq 3$, then the set of occupied vertices at any stage contains at most one vertex from each partite set. If each partite set contains an occupied vertex, then no robot can move without making a partite set containing at least $2$ robots. So it follows that $|S| \leq t-1$. On the other hand, clearly $t-1$ robots can remain in general position and visit every vertex of $K_{r_1,\dots ,r_t}$.  
		
		The second assertion now follows from~\eqref{eq:complete-multi} upon taking a complete $a$-partite graph with largest part of order $b$.
	\end{proof}
	
	We close the introduction with some examples. First, $\mob(C_4) = \mob(C_6) = 2$ and if $n\ge 3$ and $n\ne 4,6$, then $\mob(C_n) = 3$. Since $\gp(C_4) = 2$, Lemma~\ref{lem:trivial bounds} yields $\mob(C_4) = 2$. In $C_6$, a general position set of order $3$ is an independent set. But such a set is not a mobile set, hence $\mob(C_6) = 2$. Assume in the rest that $n\ge 3$ and $n\ne 4,6$. Set $V(C_n) = \mathbb{Z}_n$. When $n=2r+1$, consider the general position set $\{0,r,r+1\}$. The sequence of moves $r+1 \move r+2$, $0 \move 1$ and $r \move r+1$ keeps the property of being in general position. By iterating this procedure, each vertex of $C_n$ will be visited by a robot. As $\gp(C_n) = 3$, Lemma~\ref{lem:trivial bounds} implies that this set is a mobile gp-set. The second case to consider is when $n$ is even. Then consider a set of three vertices that are as equidistant as possible. By moving robots sequentially along the cycle in the same direction, all vertices will be visited by the robots. 
	
	Consider next the Petersen graph $P$. A scheme that allows four robots to visit every vertex of $P$ in general position is shown in Fig.~\ref{fig:Petersen4robots}; robots are initially positioned at the vertices labelled 1,2, 3 and 4. The robot at position 1 can visit vertices $a$, $b$ and $e$, the robot at $3$ can visit vertex $d$ and the robot at 2 can visit the vertices $c$ and $f$. 
	
	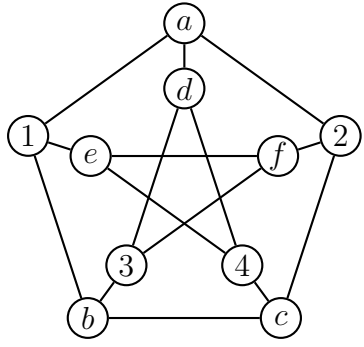
\begin{figure}[ht!]
		\centering
		\begin{tikzpicture}[x=0.2mm,y=-0.2mm,inner sep=0.2mm,scale=0.6,thick,vertex/.style={circle,draw,minimum size=15}]
			\node at (180,200) [vertex] (v1) {$a$};
			\node at (8.8,324.4) [vertex] (v2) {1};
			\node at (74.2,525.6) [vertex] (v3) {$b$};
			\node at (285.8,525.6) [vertex] (v4) {$c$};
			\node at (351.2,324.4) [vertex] (v5) {2};
			\node at (180,272) [vertex] (v6) {$d$};
			\node at (116.5,467.4) [vertex] (v7) {3};
			\node at (282.7,346.6) [vertex] (v8) {$f$};
			\node at (77.3,346.6) [vertex] (v9) {$e$};
			\node at (243.5,467.4) [vertex] (v10) {4};
			
			\path
			(v1) edge (v2)
			(v1) edge (v5)
			(v2) edge (v3)
			(v3) edge (v4)
			(v4) edge (v5)
			
			(v6) edge (v7)
			(v6) edge (v10)
			(v7) edge (v8)
			(v8) edge (v9)
			(v9) edge (v10)
			
			(v1) edge (v6)
			(v2) edge (v9)
			(v3) edge (v7)
			(v4) edge (v10)
			(v5) edge (v8)

			;
		\end{tikzpicture}
		\caption{Four robots traversing the Petersen graph in general position}
		\label{fig:Petersen4robots}
	\end{figure}
	To see that four robots is optimal, suppose for a contradiction that $K \subseteq V(G)$ is an initial configuration of at least five robots in general position. Observe that as $\alpha (P) = 4$ and $P$ is edge-transitive, we can assume that $K$ contains the two black vertices in Fig.~\ref{fig:Petersen5robots}. It is easily verified that the remaining robots must be situated on a subset of the gray vertices. If all four gray vertices are occupied, then no robot can move at all without creating three-in-a-line. If there are just five robots, then in the two independent edges joining the gray vertices one edge must have both incident vertices occupied by robots, whilst the other edge has just one robot on it. However in this configuration the only move that can be made is by the robot on the edge containing one robot and this robot can only move to the unoccupied gray vertex and back, so that not all vertices of $P$ can be visited. We have thus shown that $\mob(P) = 4$.

	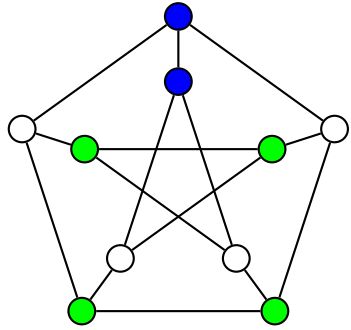
\begin{figure}[ht!]
		\centering
		\begin{tikzpicture}[x=0.2mm,y=-0.2mm,inner sep=0.2mm,scale=0.6,thick,vertex/.style={circle,draw,minimum size=10}]
			\node at (180,200) [vertex,fill=black] (v1) {};
			\node at (8.8,324.4) [vertex] (v2) {};
			\node at (74.2,525.6) [vertex,fill=gray] (v3) {};
			\node at (285.8,525.6) [vertex,fill=gray] (v4) {};
			\node at (351.2,324.4) [vertex] (v5) {};
			\node at (180,272) [vertex,fill=black] (v6) {};
			\node at (116.5,467.4) [vertex] (v7) {};
			\node at (282.7,346.6) [vertex,fill=gray] (v8) {};
			\node at (77.3,346.6) [vertex,fill=gray] (v9) {};
			\node at (243.5,467.4) [vertex] (v10) {};
			
			\path
			(v1) edge (v2)
			(v1) edge (v5)
			(v2) edge (v3)
			(v3) edge (v4)
			(v4) edge (v5)
			
			(v6) edge (v7)
			(v6) edge (v10)
			(v7) edge (v8)
			(v8) edge (v9)
			(v9) edge (v10)
			
			(v1) edge (v6)
			(v2) edge (v9)
			(v3) edge (v7)
			(v4) edge (v10)
			(v5) edge (v8)

			;
		\end{tikzpicture}
		\caption{Five robots cannot visit every vertex of the Petersen graph}
		\label{fig:Petersen5robots}
	\end{figure}
	
	Note that $\mob(G) = n(G)$ if and only if $G$ is complete. Moreover, $\mob(G) = n(G)-1$ if and only if $G$ is obtained from $K_{n-1}$ by attaching a leaf to one of its vertices. The latter result can be deduced from~\cite[Theorem 3.1]{ullas-2016} in which the graphs $G$ with $\gp(G) = n(G)-1$ are characterised; among the two families described there, the above stated graphs are the only ones with the mobile number equal to their general position number. To characterise all graphs with the mobile number equal to their general position number seems to be difficult.     
	
	\section{Graphs with cut-vertices}
	\label{sec:cut}
	
	In this section we first give a technical lemma about mobile sets in graphs with cut-vertices. Then we apply it to block graphs and to rooted products. 
	
	\begin{lemma}
		\label{lem:cut-vertex}
		Let $v$ be a cut-vertex of a (connected) graph $G$ and let $C_1, \ldots, C_k$ be the components of $G-v$. Let $G_i = G[V(C_i)\cup \{v\}]$, $i\in [k]$. If $S$ is a mobile gp-set of $G$, then the following holds. 
		\begin{enumerate}
			\item[{\rm (i)}] If $v\in S$, then $S\subseteq V(G_i)$ for some $i\in [k]$.
			\item[{\rm (ii)}] If $v\notin S$, then $S\subseteq V(C_i)\cup V(C_j)$ for some $i,j\in [k]$. Moreover, $|S\cap V(C_i)|\le 1$ and $|S\cap V(C_j)|\le \mob(G_j)$ (or the other way around). In addition, if $|S\cap V(C_i)| = 1$, then $|S\cap V(C_j)| < \mob(G_j)$. 
		\end{enumerate}
	\end{lemma}
	
	\begin{proof}
		(i) Suppose that $v\in S$, $v_j \in S\cap V(C_j)$ and $v_k \in S\cap V(C_k)$, where $j\ne k$. Then $v$, $v_j$ and $v_k$ are not in general position. 
		
		(ii) Assume $v\notin S$. Suppose that there exist vertices $v_i \in S\cap V(C_i)$, $v_j \in S\cap V(C_j)$ and $v_k \in S\cap V(C_k)$, where $|\{i,j,k\}| = 3$. Consider now the moment when the vertex $v$ is visited for the first time by some robot. At that moment we get a contradiction with (i). It follows that $S\subseteq V(C_i)\cup V(C_j)$ for some $i,j\in [k]$. By a similar argument we see that $S$ has at most one vertex in $C_i$ or $C_j$, without loss of generality assume $|S\cap V(C_i)|\le 1$. If $|S\cap V(C_i)| = 0$, then $S = S\cap V(C_j)$ and hence $S$ is in particular a mobile gp-set of $C_j$. Thus $|S| = |S\cap V(C_j)| \le \mob(G_j)$. 
		
		Assume second that $|S\cap V(C_i)| = 1$. The statement clearly holds if $|S\cap V(C_j)| = 1$, hence we may consider the case when $|S\cap V(C_j)| \ge 2$. Then the vertex $v$ must be visited by the unique vertex from $S\cap V(C_i)$, for otherwise we are in contradiction with (i). At that moment, $(S\cap V(C_j))\cup \{v\}$ is a mobile gp-set, which in turn implies that $|S\cap V(C_j)| < \mob(G_j)$.  
	\end{proof}
	
	
	
	
	
	\begin{theorem}
		\label{thm:block}
		If $G$ is a block graph, then $\mob(G) = \omega(G)$.
	\end{theorem}
	
	\begin{proof}
		Let $Q$ be a clique of $G$ with $n(Q) = \omega(G)$. Then we easily see that $V(Q)$ is a mobile gp-set, hence $\mob(G) \ge \omega(G)$. 
		
		To prove that $\mob(G) \le \omega(G)$, we proceed by induction on the number of blocks of $G$. If $G$ has only one block, then $G$ is a complete graph for which we know that $\mob(G) = n(G) = \omega(G)$. Suppose now that $B$ is an end block of $G$. Then $B$ contains exactly one cut-vertex of $G$, denote it by $v$. Let $S$ be an arbitrary mobile gp-set of $G$. If $v\in S$, then the assertion follows by Lemma~\ref{lem:cut-vertex}(i) and induction. Assume secondly that $v\notin S$. Let $H$ be the subgraph of $G$ induced by $(V(G)\setminus V(B))\cup \{v\}$. By Lemma~\ref{lem:cut-vertex}(ii) we get that $|S| \le 1 + (\mob(H) -1)$ or $|S| \le 1 + (\mob(B) -1)$. The induction assumption now gives $|S| \le 1 + (\omega (H) -1) = \omega(H) \le \omega(G)$, or $|S| \le 1 + (\omega (B) -1) = \omega(B) \le \omega(G)$. 
	\end{proof}
	
	Theorem~\ref{thm:block} clearly implies that $\mob(K_n) = n$ for $n\ge 2$ and that $\mob(T) = 2$, where $T$ is a tree of order at least $2$. 
	
	A \emph{rooted graph} is a connected graph with one chosen vertex called the \emph{root}. Let $G$ be a graph and let $H$ be a rooted graph with root $x$. The \emph{rooted product graph} $G\circ_x H$ is the graph obtained from $G$ and $n(G)$ copies of $H$, say $H_1,\dots, H_{n(G)}$, by identifying the root of $H_i$ with the $i^{\rm th}$ vertex of $G$. If $w\in V(H)$, then the vertex from $H_v$ corresponding to $w$ will be denoted by $(v,w)$.   
	
	\begin{theorem}
		\label{thm:rooted}
		If $G$ and $H$ are graphs and $x\in V(H)$, then 
		$$\max\{ \mob(G), \mob(H)\} \le \mob( G\circ_x H) \le \max\{ \mob(H), n(G)\}\,.$$  
		Moreover, the bounds are sharp. 
	\end{theorem}
	
	\begin{proof}
		Let $S$ be a mobile gp-set of $G$. Then we claim that $S$ considered as a subgraph of  $G\circ_x H$ is a mobile set of $G\circ_x H$. Indeed, if $v\in S$, then $v$ can be moved to every vertex of $H_v$ by maintaining the general position property. On the other hand, if $v\notin S$, then in the subgraph $G$ of $G\circ_x H$, some robot can move to $v$ and then continue visiting all the vertices of $H_v$. Hence $\mob( G\circ_x H) \ge |S| = \mob(G)$. 
		
		Let $S$ be a mobile gp-set of $H$. Then a copy of $S$ in an arbitrary $H_v$ is a mobile set of $G\circ_x H$. Indeed, after a robot inside $H_v$ visits the vertex $(v,x)$, this robot can freely move around $V(G\circ_x H)\setminus V(H_v)$ whilst maintaining the general position property. Thus $\mob( G\circ_x H) \ge |S| = \mob(H)$. This proves the lower bound. 
		
		Suppose now that $S\subseteq V(G\circ_x H)$, where $|S| > n(G)$, is a mobile set. By the pigeonhole principle there exists $v \in V(G)$ such that $|S \cap V(H_v)| \ge 2$. By Lemma~\ref{lem:cut-vertex} we have $|S\cap (V(G\circ_x H)\setminus V(H_v))|\le 1$. If $S\cap (V(G\circ_x H)\setminus V(H_v)) = \emptyset$, then clearly $|S| \le \mob(H)$. Otherwise, let $\{y\} = S\cap (V(G\circ_x H)\setminus V(H_v))$. Then the vertex $(v,x)$ must be visited by the robot from $y$ and at this point the robots form a mobile set of $H_v$. We conclude again that $|S| \le \mob(H)$.       
		
		Noting that $\mob(K_n \circ_x K_2) = \mob(K_2 \circ_x K_n) = n$ for $n\ge 2$, we infer that the bounds are sharp. 
	\end{proof}

	The next result shows in particular that the mobile position number of a rooted product can lie strictly between the bounds of Theorem~\ref{thm:rooted}. Let $G$ be a unicyclic graph with unique cycle $C$ of length $\ell $; we will identify the vertex set of $C$ with $\mathbb {Z}_{\ell }$ in the natural manner. Let $k \leq \ell$ be the number of vertices of $C$ that have degree at least three; we will call such a vertex a \emph{root}. If $\ell \notin \{ 4,6\}$, then it is trivial that $\mob (G) \geq 3$, as three robots can traverse $C$ in general position as described in Subsection~\ref{sec:prelim}, visiting the vertices of any pendent tree on their way. 
	
	We note firstly that if $x$ is a root of $C$, with attached trees $T_1, \dots, T_r$, then there can be at most one robot on the vertices in $\{ x\} \cup \left(\bigcup _{j=1}^{r}V(T_j)\right)$ at any time if $\mob(G) \geq 3$. By Theorem~\ref{thm:block} there can be at most two robots on $\{ x\} \cup V(T_j)$ for any $1 \leq j \leq r$. When a robot visits $x$ for the first time there will be a robot in a tree attached to $x$, say $T_1$. Since $x$ is a cut-vertex there can be no robots on the vertices of $V(G) \setminus (\{ x\} \cup V(T_1))$, so that there are only two robots on $G$. Therefore when analysing unicyclic graphs without loss of generality we can assume that $G$ is a subgraph of a sun graph, i.e. any vertex on the cycle of $G$ has degree two or three and any attached tree is a leaf. For simplicity in the following result we only deal with the case that both $\ell $ and $k$ are even; similar results are possible in the other cases by a slightly more involved argument.   
	
	\begin{theorem}
		Let $G$ be a unicyclic graph with cycle length $\ell $ such that there are $k\geq 2$ vertices of the cycle with degree at least $3$. If both $k$ and $\ell $ are even, then $\mob (G) \leq \frac{k}{2}+2$ and this is tight.
	\end{theorem}
	
	\begin{proof}
		As described previously, we can assume that $G$ is a subgraph of a sun graph. For any $i \in \mathbb{Z}_{\ell}$ we call the set $\{ i,i+1,\dots ,i+\frac{\ell }{2}-1\} $ and any attached leaves the \emph{$i$-section} of the cycle. Observe that if a robot is stationed at a vertex $x$ of $C$, then the shortest $x_1,x_2$-path containing $x$ must be of length at least $\frac{\ell }{2} + 1$, for otherwise there would be a shortest $x_1,x_2$-path in $G$ through $x$.  
		
		Suppose for a contradiction that $\mob(G) \geq \frac{k}{2}+3$. Either the $0$-section or $\frac{\ell }{2}$-section must contain $\leq \frac{k}{2}$ roots of $C$; we shall assume that the $0$-section has this property. If there are $\frac{k}{2} +2$ robots contained in the $0$-section, then $\frac{k}{2}$ of them must be positioned on leaves and the other two are on vertices of the cycle; then it is easily seen that there are three robots not in general position, possibly considering another robot from the $\frac{\ell}{2}$-section. Hence there are at most $\frac{k}{2}+1$ robots on the $0$-section. Furthermore, if there are $\frac{k}{2}+1$ robots in the $0$-section, then we can assume that there are robots stationed on leaves attached to vertices $i_1,i_2, \dots , i_{\frac{k}{2}}$ (where $0 \leq i_1 <i_2 < \dots < i_{\frac{k}{2}} \leq \frac{\ell }{2}-2$) and a robot on a vertex $i_{\frac{k}{2}+1}$ of $C$, where $i_{\frac{k}{2}} < i_{\frac{k}{2}+1} < \frac{\ell }{2}$.  
		
		Firstly, suppose that there are at least $\frac{k}{2}+4$ robots. Then as the $0$-section contains at most $\frac{k}{2}+1$ robots, there are at least three robots on the $\frac{\ell }{2}$-section. Consider the middle robot $R$ among any such three robots and suppose this robot is at the vertex $y$ or at a leaf attached to $y$. Then $R$ must be stationed on the leaf attached to $y$, otherwise it is on a shortest path between the other two vertices of the $\frac{\ell }{2}$-section. Then no robot can visit the vertex $y$, a contradiction.
		
		Now suppose that there are exactly $\frac{k}{2}+3$ robots. If any section contains less than $\frac{k}{2}$ roots, then the above argument yields a contradiction, so we can assume that every section of $C$ contains exactly $\frac{k}{2}$ roots. This implies that if $x,x'$ is any pair of antipodal vertices on $C$, then either both $x,x'$ are roots or neither $x,x'$ are roots. As the $0$-section of $G$ contains $\frac{k}{2}$ roots, there must be at least two robots $R_1$ and $R_2$ at vertices $x_1,x_2$ of $C$ or attached leaves in the $\frac{\ell }{2}$-section. If $R_1$ is stationed on a leaf at some point it must descend to a vertex of $C$ in order to visit the root, so we can assume that $R_1$ is on $C$. When this occurs, we consider the two robots at shortest distance from $R_1$ on either side of $R_1$ with respect to the cycle. Since they must be at distance bigger than $\frac{\ell }{2}$, at least one of then must be in the $0$-section. As the $0$-section can hold at most $\frac{k}{2}+1$ robots, we can assume that $R_1$ is on $C$ and $R_2$ is on a leaf attached to $x_2$ in the $\frac{\ell }{2}$-section. Now by the preceding argument the vertex $x_2'$ antipodal to $x_2$ on $C$ is also a root in the $0$-section; as there are $\frac{k}{2}$ roots and $\frac{k}{2} + 1$ robots in the $0$-section, there must also be a robot $R_2'$ on the leaf attached to $x_2'$. However, this implies that the robot $R_1$ lies on a shortest path between the robots $R_2$ and $R_2'$. As a conclusion we get that $\mob (G) \leq \frac{k}{2}+2$. 
		
		We now show that this bound is tight. For $\ell \geq k$ we define the {\em $(\ell,k)$-jellyfish} to be the unicyclic graph with order $\ell+k$ formed from an $\ell $-cycle $C_{\ell }$ (with vertex set identified with $\mathbb{Z}_{\ell }$) with leaves attached to the vertices $0,1,\dots ,k-1$. We will denote the leaf attached to vertex $i$ by $i'$. We describe how $\frac{k}{2}+2$ vertices can visit every vertex of the jellyfish whilst staying in general position. We begin with robots at the vertices $0',1',\dots ,(\frac{k}{2}+1)'$ (call the robots $R_0,R_1$ etc). The first robot $R_0$ makes the move $0' \move 0$ and then moves around $C_{\ell }$ in the direction $0 \move \ell-1 \move \ell-2 \move \cdots \move \frac{k}{2}+2$. When robot $R_0$ visits vertices $k-1,k-2,\dots ,\frac{k}{2}+2$ it can also visit the attached leaves $(k-1)',(k-2)',\dots ,(\frac{k}{2}+2)'$. Finally robot $R_0$ moves to the leaf $(\frac{k}{2}+2)'$. We now send robot $R_1$ around the cycle in the same direction and station it at leaf $(\frac{k}{2}+3)'$ and in general for $1 \leq i \leq \frac{k}{2}-3$ we send robot $R_i$ around the cycle to leaf $(\frac{k}{2}+2+i)'$ in sequence. At this point there are robots at the leaves $(\frac{k}{2}-2)',(\frac{k}{2}-1)',\dots ,(k-1)'$ and all vertices of $G$ have been visited with the exception of the leaves $(\frac{k}{2}-2)',\dots,(\frac{k}{2}+1)'$. Now robot $R_{\frac{k}{2}-2}$ moves $(\frac{k}{2}-2)' \move \frac{k}{2}-2$ and moves around the cycle in the same direction \[\frac{k}{2}-2\move \frac{k}{2}-3 \move \cdots 0 \move \ell -1 \move \cdots \move k,\] stopping at vertex $k$. Next, robot $R_{\frac{k}{2}-1}$ performs the move $(\frac{k}{2}-1)' \move \frac{k}{2}-1$. Finally, by symmetry, it follows that vertices $\frac{k}{2}$ and $\frac{k}{2}+1$ can also be visited.
	\end{proof}
	Note that if $\ell = k$ then the graph in question is a rooted product.

	\section{Kneser graphs and line graphs of complete graphs}
	\label{sec:kneser}
	
	If $n\ge 2k$, then the {\em Kneser graph} $K(n,k)$ has all $k$-subsets of $[n]$ as vertices, two vertices being adjacent if the corresponding sets are disjoint. In~\cite[Theorem 2.2]{ghorbani-2021} it was proved that $\gp(K(n,2)) = 6$ for $n\in \{4,5,6\}$ and $\gp(K(n,2)) = n-1$ for $n\ge 7$. For additional results on the gp-number of Kneser graphs see~\cite{patkos-2020}. 
	
	The Kneser graph $K(5,2)$ is the Petersen graph for which we have seen in Section~\ref{sec:prelim} that $\mob(K(5,2)) = 4$. This fact generalises as follows. 
	
	\begin{theorem}
		\label{thm:kneser}
		If $n\geq 5$, then $\mob(K(n,2)) = \max \{ 4,\left \lfloor \frac{n-3}{2} \right \rfloor\} $.
	\end{theorem}
	
	\begin{proof}
		For $n \geq 5$ the diameter of $K(n,2)$ is $2$. It follows from~\cite[Theorem 4.1]{AnaChaChaKlaTho} that a set $S$ of vertices of $K(n,2)$ is in general position if and only if it is an independent union of cliques. Moreover, from the proof of~\cite[Theorem 2.2]{ghorbani-2021} we recall that the structure of $S$ is one of the following: (i) $S$ consists of a clique of order at least $3$ (ii) the largest clique of $S$ is of order $2$, in which case $|S|\le 6$; and (iii) $S$ induces an independent set. We now discuss these structures in turn.
		
		\medskip\noindent
		\textbf{Case 1}: $S$ contains a clique of order at least $3$.\\ 
		We can assume that $S$ is of the form $\{ \{ 1,2\} ,\{3,4\} ,\dots ,\{ 2r-1,2r\}\} $ for some $r \geq 3$. If $2r \geq n-1$, then none of the robots can move, whereas if $2r = n-2$, then any robot of $S$ can only move to the vertex $\{ n-1,n\} $ and back again, so that not every vertex can be visited. 
		
		However, if $2r \leq n-3$ then $S$ is a mobile gp-set. There are three forms of vertex that need to be visited: (i) $\{ a,b\} $, where $a,b \notin [2r]$, (ii) $\{ c,d\}$, where $c \in [2r]$ and $d \notin [2r]$ and (iii) $\{ e,f\} $, where $e,f \in [2r]$, but $\{ e,f\} \notin S$. Without loss of generality we can assume that $\{ a,b\} = \{ n,n-1\} $, $\{ c,d\} = \{ 1,n\} $ and $\{ e,f\} = \{ 1,3\} $. These vertices can be visited by the following sequences of moves.
		
		Case (i): $\{ 1,2\} \move \{ n-1,n\}$.
		Case (ii): $\{ 1,2\} \move \{ n-2,n-1\} \move \{ 1,n\} $.
		Case (iii): The robot at vertex $\{ 1,2\} $ moves according to $\{ 1,2\} \move \{ n-1,n\} \move \{ 2,n-2\} $, so that the robots now occupy the set $\{ \{3,4\} ,\{5,6\} ,\dots, \{2r-1,2r\} \} \cup \{ \{ 2,n-2\} \}$. Now the robot at $\{ 3,4\} $ makes the following moves: $\{ 3,4\} \move \{ n-1,n\} \move \{ 1,3\} $.
		
		In summary, if $S$ contains a clique of order at least $3$, then since $|S|=r$ and $r\le (n-3)/2$, we deduce that $|S| \le \lfloor (n-3)/2\rfloor$. The aforementioned procedure also shows that a mobile general position set of cardinality $\lfloor (n-3)/2\rfloor$ exists, hence $\mob(K(n,2)) \ge \lfloor (n-3)/2 \rfloor$. 
		
		\medskip\noindent
		\textbf{Case 2}: $S$ contains an induced clique of order $2$ (and no triangle).\\ 
		By the result of~\cite{ghorbani-2021} if $S$ contains an induced copy of $K_2$, say on the vertices $\{ 1,2\} ,\{ 3,4\} $, then all vertices of $S$ are subsets of $\{ 1,2,3,4\} $. Firstly we show that there is such a mobile set with four robots. For any distinct $a,b,c,d \in [n]$ if robots are positioned at the vertices $\{ a,b\}$, $\{ c,d\} $, $\{ a,c\} $ and $\{ a,d\} $ then by the moves $\{ a,c\} \move \{ b,d\} $ and $\{ a,d\} \move \{ b,c\} $ the robots can visit every vertex that is a subset of $\{ a,b,c,d\} $ whilst remaining in general position. Suppose that we start with the robots at $\{ 1,2\} $, $\{ 3,4\} $, $\{ 1,3\} $ and $\{ 1,4\}$. Let $a,b \notin \{ 1,2,3,4\} $. The move $\{ 1,2\} \move \{ 3,a\} $ transforms the set of occupied vertices into a general position set of the same form, so that all subsets of $\{ 1,3,4,a\} $ can be visited. Furthermore, starting with robots at $\{ 1,2\} $ ,$\{ 3,4\} $, $\{ 1,3\} $ and $\{ 1,4\}$ the sequence of moves $\{ 1,2\} \move \{ 3,a\}$, $\{ 3,4\} \move \{ 1,a\} $ and $\{ 1,4\} \move \{ a,b\} $ allows any vertex of the form $\{ a,b\} $ to be visited.
		
		Suppose for a contradiction that there exists such a mobile set with $|S| \geq 5$. At some point a robot has to move to a vertex $\{ a,b\} $ that is not a subset of $\{ 1,2,3,4\} $; we can assume that immediately before this step there are robots on the vertices $\{ 1,2\} $, $\{ 3,4\} $, $\{ 1,3\}$, $\{ 2,4\} $ and $\{ 1,4\} $ (and possibly $\{ 2,3\} $) and we let $S'$ be the set of occupied vertices immediately after this step. $S'$ cannot contain any induced copy of $K_2$ on subsets of $\{ 1,2,3,4\} $ (otherwise we would have $\{ a,b\} \subset \{ 1,2,3,4\} $) and we can also assume by Case 1 that $S'$ does not contain a clique of order at least 3. Clearly this is impossible.   
		
		In summary, if $S$ contains an induced $K_2$ and no triangle, then $|S|\le 4$. Moreover, in this case we also have $\mob(K(n,2)) \ge 4$. 
		
		\medskip\noindent
		\textbf{Case 3}: $S$ is an independent set.\\ 
		There are two possible structures of independent set in $K(n,2)$. Either $S$ is of the form $\{ \{ 1,2\}, \{ 1,3\}, \{ 2,3\}\}$, or else is of the form $\{ \{1,2\}, \{1,3\}, \dots, \{1,r\}\}$ for some $r \in [n]$. In the first case $|S| \leq 3$. Suppose that $|S| = r-1 \geq 4$. Let $a,b \notin [r]$. Without loss of generality we can assume that the robot at $\{ 1,2\} $ makes the first move. Without loss of generality there are three types of move that the robot can make: (i) $\{ 1,2\} \move \{ a,b\} $, (ii) $\{ 1,2\} \move \{3,a\}$ and (iii) $\{ 1,2\} \move \{ 3,4\} $. In Cases (i) and (ii) $\{ 1,3\} \sim \{ a,b\} \sim \{ 1,4\} $ and $\{ 1,4\} \sim \{ 3a\} \sim \{ 1,5\} $ respectively would be shortest paths containing three robots, a contradiction. For Case (iii), if $r \geq 6$, then $\{ 1,5\} \sim \{ 3,4\} \sim \{1,6\} $ shows that there would be three robots in a line, whereas if $r = 5$ this move returns us to Case 2 above. 
		
		All possibilities have been considered, hence $\mob (K(n,2)) = \max \{ 4,\left \lfloor \frac{n-3}{2} \right \rfloor\}$ holds for $n \geq 5$.
	\end{proof}
	
	We now determine the mobility number of the complement of the Kneser graphs $K(n,2)$, i.e.\ the line graph $L(K_n)$ of $K_n$. Recall that the line graph $L(G)$ of a graph $G$ has $V(L(G)) = E(G)$, vertices being adjacent if the corresponding edges are incident in $G$. By the result of~\cite{ghorbani-2021} the general position number of this graph is $n$ if $3|n$ and $n-1$ otherwise; we will show that the mobile gp-number of these graphs is very close to the general position number.
	
	\begin{theorem}
		\label{thm:line-graphs}
		If $n \geq 4$, then $\mob (L(K_n)) = n-2$.
	\end{theorem}
	\begin{proof}
		Let $S$ be a largest mobile set in $L(K_n)$. By~\cite[Theorem 3.1]{AnaChaChaKlaTho} each component of the subgraph induced by $S$ is a clique. Calls these cliques $W_1,\dots, W_k$, $k \geq 1$, with orders $n_1, \dots, n_k$ respectively. Each of these cliques correspond to either an induced star in $K_n$ or an induced $C_3$. We identify the vertex set of $K_n$ with $[n]$ and a vertex of $L(K_n)$, that is an edge $ij$ of $K_n$, with the pair $\{ i,j\} $. 
		
		Suppose that $S$ contains a clique $W$ that corresponds to a $C_3$ in $K_n$, so that without loss of generality $W$ is the clique on the edges $\{ 1,2\} $, $\{ 2,3\} $ and $\{ 1,3\} $. No robot can be positioned at an edge $\{ 1,i\} $, where $4 \leq i \leq n$. Consider the edge $\{ 1,4\} $. Observe that no robot on an edge of $W$ can move to $\{ 1,4\} $ without creating three-in-a-line. Similarly any robot on an edge $\{ i,j\} $, $4 \leq i,j \leq n$, would create three-in-a-line if it moves to $\{ 1,4\} $. Therefore no robot can ever visit the edge $\{ 1,4\} $ in this scenario, a contradiction, so we can assume that each clique in $S$ corresponds to an induced star in $K_n$. 
		
		Following the convention of~\cite{ghorbani-2021} for $1 \leq i \leq k$ we write \[ X_i = \bigcup _{\{ i,j\} \in V(W_i)}\{ i,j\} \] and set $x_i= |X_i|$. As each $W_i$ corresponds to a star of $K_n$ we have $x_i = n_i+1$ for $1 \leq i \leq k$. It follows that
		\[ |S| = \Sigma _{i=1}^{k}n_i = \Sigma _{i=1}^{k}(x_i-1) \leq n-k.\]
		Thus $|S| \leq n-1$ and we have equality if and only if $k = 1$ and $W_1$ corresponds to an induced star of order $n$ in $K_n$; however in this case no robot is free to move without violating the general position property. Therefore $|S| \leq n-2$.
		
		Conversely, there is a mobile set of $L(K_n)$ of order $n-2$, namely \[ \{ \{1,2\} ,\{ 1,3\} ,\dots ,\{ 1,n-1\} \} .\] The move $\{ 1,2\} \move \{ 1,n\} $ is valid, so all edges adjacent to $1$ can be visited. Also for $2 \leq i \leq n-1$ the move $\{ 1,i\} \move \{ n,i\} $ is valid. Therefore the only vertices that must still be visited are those of the form $\{ i,j\} $, where $2 \leq i \leq j$; without loss of generality we show how to visit $\{ 2,3\} $. This can be done by performing the move $\{ 1,i\} \move \{ i,n\} $ for $4\leq i \leq n-1$, followed by $\{ 1,2\} \move \{ 2,3\} $. This completes the proof.
	\end{proof}

	\section{Concluding remarks}
	
	As a conclusion we list a few interesting open problems that arise naturally. 
	
	\begin{itemize}
		\item Since it not even clear whether checking if a given set of vertices of a graph is a mobile general position set is in NP, the computational complexity of computing the mobile general position number seems to be a challenging and interesting problem.  
		\item Determine the mobile general position number for all unicyclic graphs. 
		\item Determine $\mob(K(n,k))$ for $k\ge 3$.
		\item Based on Theorem~\ref{thm:line-graphs} it would be interesting to investigate the mobile general position number of arbitrary line graphs.
		\item In view of Theorems~\ref{thm:kneser} and~\ref{thm:line-graphs}, 
		since $L(K_n)$ is the complement of $K(n,2)$, we propose to investigate $\mob(G) + \mob(\overline{G})$ for an arbitrary $G$, that is, the additive Nordhaus-Gaddum inequalities. 
		\item Is there a general relationship between the mobile general position number and clique number?
		\item Finally, in our model it suffices that each vertex is visited by one of the robots. However, possible applications can also be imagined in which each vertex must be visited by every robot while still maintaining the general position property at all times. We believe this variant of mobility deserves independent investigation. 
	\end{itemize}
	
	\section*{Acknowledgments}
	
	Sandi Klav\v{z}ar was partially supported by the Slovenian Research Agency (ARRS) under the grants P1-0297, J1-2452, and N1-0285. Ismael G. Yero has been partially supported by the Spanish Ministry of Science and Innovation through the grant PID2019-105824GB-I00. Moreover, this investigation was developed while James Tuite and Ismael Yero were visiting the University of Ljubljana and these authors thank the university for their hospitality; Ismael Yero received support for this visit from the ``Ministerio de Educaci\'on, Cultura y Deporte'', Spain, under the ``Jos\'e Castillejo'' program for young researchers (reference number: CAS21/00100). James Tuite also gratefully acknowledges funding support from EPSRC grant EP/W522338/1. Aditi Krishnakumar received funding to work on this project as part of a research internship at the Open University.
	
	Finally the authors would especially like to thank Sumaiyah Boshar and Benjamin Wilkins for their help on this project as part of a Nuffield research internship at the Open University.
	
	\section*{Conflict of interests}
	The authors declare that there is no conflict of interests regarding the publication of this paper.

\end{document}